\documentclass[10pt]{amsart}

\usepackage{amsmath,amssymb,mathtools}
\usepackage{microtype}
\usepackage[colorlinks=true,linkcolor=blue,citecolor=blue,urlcolor=blue]{hyperref}

\newtheorem{theorem}{Theorem}[section]
\newtheorem{corollary}[theorem]{Corollary}
\newtheorem{proposition}[theorem]{Proposition}
\newtheorem{lemma}[theorem]{Lemma}

\DeclareMathOperator{\covol}{covol}
\newcommand{\R}{\mathbb{R}}
\newcommand{\Z}{\mathbb{Z}}
\newcommand{\calS}{\mathcal{S}}
\newcommand{\wh}{\widehat}
\newcommand{\norm}[1]{\lVert #1\rVert}
\newcommand{\abs}[1]{\lvert #1\rvert}
\newcommand{\ip}[2]{\langle #1,#2\rangle}

\title[Saturation and No-Go Theorems for Scalar Poisson Certificates]{Saturation and
No-Go Theorems for Scalar Poisson Certificates of Gaussian Mass Maximality}

\author[Scott Duke Kominers]{Scott Duke Kominers}

\address{Harvard Business School; Department of Economics and Center of
Mathematical Sciences and Applications, Harvard University; and a16z crypto}

\email{kominers@fas.harvard.edu}

\thanks{I used LLMs to assist with computations, analysis, and synthesis in
the preparation of this article, particularly GPT-5.5 Pro and Claude 4.7 Opus
(accessed in part via Poe with the support of Quora, where I am an advisor).
I particularly appreciate helpful comments from Noam Elkies, as well as a
review from Refine.ink. The problem, methods, and eventual written form are my
own; and of course any errors remain my responsibility. This work was conducted
while I was visiting the Technological Innovation, Entrepreneurship, and
Strategic Management (TIES) Group at the MIT Sloan School of Management; I
greatly appreciate their hospitality.}


\begin{document}

\begin{abstract}
Regev and Stephens-Davidowitz conjectured that the Gaussian mass
$\Theta_\Lambda(t) = \sum_{x \in \Lambda} e^{-t\norm{x}^2}$ of any integral
lattice $\Lambda \subset \R^n$ is bounded above by $\Theta_{\Z^n}(t)$. For
$n\ge 4$, we prove a saturation theorem for the natural scalar
Poisson-summation certificates of this conjecture: any such certificate that is
sharp at $\Z^n$ must interpolate the Gaussian, and have vanishing Fourier
transform, at every nonzero point of integer squared norm. Applied to the
lattice $E_8 \oplus \Z^{n-8}$, this rigidity is incompatible with the strict
theta-series gap
$\Theta_{\Z^8}(t) - \Theta_{E_8}(t) = \theta_2(it/\pi)^4\,\theta_4(it/\pi)^4 > 0$.
Consequently, in dimensions $n \ge 8$, no scalar Poisson certificate can attain
the sharp $\Z^n$ Gaussian mass bound. The same argument rules out the
corresponding scalar certificate strategy for the stable-lattice formulation of
the conjecture, and extends to orbit-constant graded families
$\Lambda \mapsto h_\Lambda$; near-sharp sequences are similarly excluded under
a uniform summability hypothesis.
\end{abstract}

\maketitle

\section{Introduction}

Let $\Lambda\subset\R^n$ be a full-rank lattice. We call $\Lambda$
\emph{integral} if $\ip{x}{y}\in\Z$ for all $x,y\in\Lambda$, and
\emph{unimodular} if $\covol(\Lambda)=1$. The \emph{Gaussian mass}
\[
  \Theta_\Lambda(t)=\sum_{x\in\Lambda} e^{-t\norm{x}^2}\qquad (t>0)
\]
is a smoothed lattice-point count, recording the distribution of vectors of
$\Lambda$ across Euclidean shells. Regev and Stephens-Davidowitz~\cite{RegevSD}
proved a nearly-sharp reverse Minkowski theorem for integral lattices and
raised the natural conjecture that
\begin{equation}\label{eq:RSD}
  \Theta_\Lambda(t)\le \Theta_{\Z^n}(t)\qquad(t>0)
\end{equation}
for every integral lattice $\Lambda\subset\R^n$ \cite{RegevSD,RegevSDoriginal}.
The Gaussian formulation is important: the analogous shell-counting statement is
false. Already in dimension~$8$, counting the origin, $E_8$ has $241$ vectors
of squared norm at most $2$, whereas $\Z^8$ has $129$. Thus any proof of
\eqref{eq:RSD} must use cancellations among different shells, not
shell-by-shell domination. The discrete shell-counting analogue of
\eqref{eq:RSD}, proven in the same paper~\cite{RegevSD}, was recently shown by
Kominers~\cite{Kominers} to be saturated for $n \ge 2$ only at $(\Z^n, k=1)$
and $(E_8, k=2)$; these two lattices also play a distinguished role in the
present note.

One direct Fourier-analytic approach to \eqref{eq:RSD} is modeled on the
Cohn--Elkies linear programming bound for sphere packing \cite{CohnElkies} and
the magic-function/interpolation method behind universal optimality phenomena
for spheres and Euclidean lattices \cite{CohnKumar,CKMRV,Viazovska}. With the
Fourier transform normalized by
\begin{equation}\label{eq:fourier}
  \wh f(\xi)=\int_{\R^n} f(x)e^{-2\pi i\ip{x}{\xi}}\,dx,
\end{equation}
Poisson summation says
\begin{equation}\label{eq:PS-intro}
  \sum_{x\in\Lambda} f(x)=\frac{1}{\covol(\Lambda)}
  \sum_{\xi\in\Lambda^*}\wh f(\xi),
\end{equation}
where $\Lambda^*=\{\xi\in\R^n: \ip{x}{\xi}\in\Z\text{ for all }x\in\Lambda\}$
is the \emph{dual lattice}. If $\Lambda$ is integral, then
$\Lambda \subseteq \Lambda^*$. If moreover $\covol(\Lambda)=1$, then
$\covol(\Lambda^*)=1$ as well, and hence $\Lambda=\Lambda^*$ as subsets of
$\R^n$, i.e., $\Lambda$ is \emph{self-dual}. Thus a single
lattice-independent Schwartz function $h$ satisfying
\[
  h(x)\ge e^{-t\norm{x}^2},
  \qquad
  \wh h(x)\le 0
\]
at every nonzero $x$ in every unimodular integral lattice
$\Lambda \subset \R^n$ would give, by Poisson summation,
\[
  \Theta_\Lambda(t)
   \le 1+\wh h(0)-h(0).
\]
We call such an $h$ a \emph{scalar Poisson certificate}, and we say $h$ is
\emph{sharp (at $\Z^n$)} if the resulting bound is tight at $\Z^n$, i.e., if
\begin{equation}\label{eq:sharp-normalization}
  1+\wh h(0)-h(0)=\Theta_{\Z^n}(t).
\end{equation}

The main result of this note is a rigidity theorem showing that sharpness
forces strong pointwise equality on every integral shell.

\begin{theorem}[Integral-shell saturation]\label{thm:saturation}
Let $n\ge 4$ and $t>0$, and let $h:\R^n\to\R$ be an even Schwartz function
satisfying:
\begin{enumerate}
\item[\textup{(i)}] $h(x)\ge e^{-t\norm{x}^2}$ for every nonzero $x$ in every
unimodular integral lattice $\Lambda\subset\R^n$;
\item[\textup{(ii)}] $\wh h(x)\le 0$ for every nonzero $x$ in every unimodular
integral lattice $\Lambda\subset\R^n$.
\end{enumerate}
If $h$ is sharp at $\Z^n$ in the sense of~\eqref{eq:sharp-normalization}, then
\begin{equation}\label{eq:saturation-conclusion}
  h(x)=e^{-t\norm{x}^2}
  \quad\text{and}\quad
  \wh h(x)=0
  \qquad(x\ne 0,\ \norm{x}^2\in\Z).
\end{equation}
\end{theorem}

For $n\ge 4$, the union of the nonzero point sets of the rotations $U\Z^n$ is
precisely the set of nonzero $x\in\R^n$ with $\norm{x}^2\in\Z$. Thus, in the
unimodular integral setting, the pointwise sign hypotheses amount to
constraints on the integer-squared-radius shells. The dimension restriction
$n \ge 4$ enters only through Lagrange's four-square theorem, which is what
carries saturation from $\Z^n \setminus \{0\}$ to every sphere of positive
integer squared radius.

In dimensions $n \ge 8$, the rigidity conclusion is incompatible with the
lattice $E_8 \oplus \Z^{n-8}$, which is unimodular and integral, has integral
shells, and has theta function strictly smaller than $\Theta_{\Z^n}$. This
yields the main ``no-go'' consequence ruling out \eqref{eq:sharp-normalization}.

\begin{corollary}[Strictness of scalar Poisson certificates]\label{cor:strict}
Let $n\ge 8$ and $t>0$, and suppose $h:\R^n\to\R$ is an even Schwartz function
satisfying conditions \textup{(i)} and \textup{(ii)} of
Theorem~\ref{thm:saturation}. Then
\begin{equation}\label{eq:strict}
  1+\wh h(0)-h(0)>\Theta_{\Z^n}(t).
\end{equation}
\end{corollary}

\begin{corollary}[No sharp scalar Poisson certificate]\label{cor:nogo}
For $n\ge 8$ and $t>0$, there is no even Schwartz function $h:\R^n\to\R$
satisfying conditions \textup{(i)} and \textup{(ii)} of
Theorem~\ref{thm:saturation} together with the sharp
normalization~\eqref{eq:sharp-normalization}.
\end{corollary}

The focus on $n\ge 8$ in Corollaries~\ref{cor:strict} and~\ref{cor:nogo}
arises because $n=8$ is the first nontrivial dimension for the unimodular
integral comparison---for $n<8$, every positive definite unimodular integral
lattice of rank $n$ is isometric to $\Z^n$. Meanwhile, the evenness assumption
is without loss of generality: for any real Schwartz $h$ satisfying conditions
\textup{(i)} and \textup{(ii)} at every lattice point, the even part
$\tfrac{1}{2}(h(x)+h(-x))$ also satisfies them, since every lattice is
centrally symmetric. Evenness simply ensures that $\wh h$ is real, so that
condition~\textup{(ii)} is unambiguous.

The obstruction is caused by equality, not by limited flexibility in choosing
$h$. Sharpness at every rotation of $\Z^n$ would force a would-be certificate
to agree with the Gaussian on every shell of integer squared radius; the
resulting interpolation conditions are incompatible with the theta series of
$E_8$.

The same saturation mechanism extends the obstruction substantially beyond the
integral unimodular setting. In Section~\ref{sec:extensions} we prove three
strengthenings:
\begin{itemize}
\item The obstruction holds for the broader class of \emph{stable lattices} in
which Regev and Stephens-Davidowitz~\cite{RegevSDoriginal} originally
formulated~\eqref{eq:RSD} (Corollary~\ref{cor:stable-strict}).
\item Moreover, the obstruction holds for any \emph{orbit-constant graded
family} $\Lambda \mapsto h_\Lambda$ in which each lattice is allowed its own
certificate, provided the assignment $\Lambda \mapsto h_\Lambda$ is constant on
$O(n)$-orbits (Theorem~\ref{thm:graded-strict}). This rules out, in particular,
every certificate framework in which $h_\Lambda$ is determined by
orientation-independent lattice invariants such as the theta function, the
shell-count sequence, modular-form coordinates for even unimodular $\Lambda$,
or any finite list of $O(n)$-invariants.
\item And the obstruction rules out \emph{near-sharp sequences} of
certificates, under a uniform-summability hypothesis on the restrictions to
$E_8 \oplus \Z^{n-8}$ (Corollary~\ref{cor:no-compact-limits}). Limiting
approaches of this scalar form must therefore be genuinely noncompact.
\end{itemize}

\subsection*{Related literature}
To the best of our knowledge, the particular saturation obstruction we identify
has not been isolated previously. The closest conceptual neighbors are results
about forced zeros or limitations in the Cohn--Elkies linear program.
Zubrilina~\cite{Zubrilina} studied the possible zero sets of optimal
Cohn--Elkies functions, while Cohn--Triantafillou \cite{CohnTriantafillou},
Li~\cite{Li}, and de Courcy-Ireland--Dostert--Viazovska~\cite{DCIDV} proved
nonsharpness results for that linear program in several dimensions different
from $1,2,8,24$. The mechanism here is different: the obstruction is not a
numerical failure of a linear program, but an exact saturation phenomenon caused
by the coexistence of $\Z^8$ and $E_8$. Our result is also targeted: we rule
out the direct Cohn--Elkies/CKMRV-style scalar Poisson certificate even before
one confronts nonunimodular integral lattices---but we do not rule out
modular-form or theta-series parametrizations as in the Belfiore--Sol\'e line
and in recent work of Bollauf--Lin at fixed argument
\cite{BelfioreSole,OggierSoleBelfiore,ErnvallHytonen,BollaufLin}, or
higher-order methods such as semidefinite programming relaxations that exploit
the integrality of pairwise inner products.

\section{Preliminaries}

We use the Fourier transform convention displayed in~\eqref{eq:fourier}. For
$f\in\calS(\R^n)$, Poisson summation in the form \eqref{eq:PS-intro} is
absolutely convergent. In particular, if $\Lambda$ is unimodular and integral,
then $\Lambda\subseteq\Lambda^*$ and both lattices have covolume $1$, hence
$\Lambda=\Lambda^*$ and
\begin{equation}\label{eq:PS-unimodular}
  \sum_{x\in\Lambda} f(x)=\sum_{x\in\Lambda}\wh f(x).
\end{equation}
Orthogonal direct sums have multiplicative theta series:
\begin{equation}\label{eq:multiplicative-theta}
  \Theta_{\Lambda_1\oplus\Lambda_2}(t)
  =\sum_{x\in\Lambda_1}\sum_{y\in\Lambda_2}e^{-t(\norm{x}^2+\norm{y}^2)}
  =\Theta_{\Lambda_1}(t)\Theta_{\Lambda_2}(t).
\end{equation}

Let $\tau=it/\pi$ and put $q=e^{\pi i\tau}=e^{-t}$. We use the Jacobi theta
nullwerte
\[
\theta_2(\tau)=\sum_{m\in\Z}q^{(m+1/2)^2},\qquad
\theta_3(\tau)=\sum_{m\in\Z}q^{m^2},\qquad
\theta_4(\tau)=\sum_{m\in\Z}(-1)^m q^{m^2}.
\]
Some references work with $Q=e^{2\pi i\tau}=q^2$ rather than $q$, particularly
for the Eisenstein series $E_4(\tau)=1+240\sum_{m\ge 1}\sigma_3(m)Q^m$. Both
conventions are used in the sequel; the relations $Q=q^2$ and $\tau=it/\pi$
allow either to be specialized to the real parameter $t>0$. With this
convention, $\Theta_{\Z^8}(t)=\theta_3(\tau)^8$. The $E_8$ root lattice is the
even unimodular rank-$8$ lattice
\[
  E_8=D_8\cup\bigl(D_8+(1/2,\ldots,1/2)\bigr),
  \qquad D_8=\{x\in\Z^8: x_1+\cdots+x_8\equiv 0\!\!\!\!\pmod 2\},
\]
up to isometry. Since the theta series of an even unimodular lattice of rank
$8$ is a weight-$4$ modular form for $\mathrm{SL}_2(\Z)$ with constant term
$1$, and since the space of such forms is spanned by the Eisenstein series
$E_4$, we have
\begin{equation}\label{eq:E8E4}
  \Theta_{E_8}(t)=E_4(it/\pi).
\end{equation}
Glaisher's identity and Jacobi's abstruse identity \cite[Ch.\ 4]{ConwaySloane}
give
\begin{equation}\label{eq:glaisher-jacobi}
  E_4(\tau)=\frac{1}{2}\bigl(\theta_2(\tau)^8+
  \theta_3(\tau)^8+\theta_4(\tau)^8\bigr),
  \qquad
  \theta_3(\tau)^4=\theta_2(\tau)^4+\theta_4(\tau)^4.
\end{equation}
Combining \eqref{eq:E8E4} and \eqref{eq:glaisher-jacobi}, and writing
$A=\theta_2(\tau)^4$ and $B=\theta_4(\tau)^4$, we obtain
\begin{align}\label{eq:Z8minusE8}
  \Theta_{\Z^8}(t)-\Theta_{E_8}(t)
  &=\theta_3^8-\frac12(\theta_2^8+\theta_3^8+\theta_4^8)\nonumber\\
  &=\frac12\bigl((A+B)^2-A^2-B^2\bigr)
   =AB
   =\theta_2(\tau)^4\theta_4(\tau)^4.
\end{align}
For $t>0$, the inequalities $\theta_2(\tau)>0$ and $\theta_3(\tau)>0$ are
immediate from the series for $\theta_2$ and $\theta_3$, while
$\theta_4(\tau)>0$ follows, for example, from the product formula
\[
  \theta_4(\tau)=\prod_{m\ge 1}(1-q^{2m})(1-q^{2m-1})^2
  \qquad (0<q<1).
\]
We record the resulting comparison as a lemma.

\begin{lemma}\label{lem:e8-gap}
For every $t>0$,
\begin{equation}\label{eq:positive-gap}
  \Theta_{\Z^8}(t)-\Theta_{E_8}(t)=\theta_2(it/\pi)^4\theta_4(it/\pi)^4>0,
\end{equation}
and, consequently, for every $n\ge 8$,
\[
  \Theta_{E_8\oplus\Z^{n-8}}(t)<\Theta_{\Z^n}(t).
\]
\end{lemma}

\section{The Saturation Theorem}\label{sec:saturation}

We now prove Theorem~\ref{thm:saturation}. The proof uses two elementary
observations. First, if $U\in O(n)$, then $U\Z^n$ is again a unimodular
integral lattice, since covolume and inner products are preserved by $U$;
moreover,
\begin{equation}\label{eq:sametheta}
  \Theta_{U\Z^n}(t)=\Theta_{\Z^n}(t).
\end{equation}
Second, Lagrange's four-square theorem states that every nonnegative integer is
a sum of four integer squares; since $n\ge 4$, embedding $\Z^4\hookrightarrow
\Z^n$ in any four coordinates shows that for every $m\in\Z_{>0}$ there is a
vector $z\in\Z^n$ with $\norm{z}^2=m$.

\begin{proof}[Proof of Theorem~\ref{thm:saturation}]
Suppose that $h\in\calS(\R^n)$ satisfies conditions \textup{(i)} and
\textup{(ii)} and is sharp at $\Z^n$. By Poisson summation applied to the
unimodular lattice $U\Z^n$ for an arbitrary $U\in O(n)$, together with
\eqref{eq:sametheta}, condition \textup{(i)} on $U\Z^n$, and condition
\textup{(ii)} on $U\Z^n$, we obtain the chain
\begin{align}\label{eq:saturation-chain}
  \Theta_{\Z^n}(t)-1
  &=\sum_{z\in U\Z^n\setminus\{0\}} e^{-t\norm{z}^2} \nonumber\\
  &\le \sum_{z\in U\Z^n\setminus\{0\}} h(z) \nonumber\\
  &=\sum_{z\in U\Z^n}h(z)-h(0) \nonumber\\
  &=\sum_{z\in U\Z^n}\wh h(z)-h(0) \nonumber\\
  &=\wh h(0)-h(0)+\sum_{z\in U\Z^n\setminus\{0\}}\wh h(z) \nonumber\\
  &\le \wh h(0)-h(0)\nonumber\\
  &=\Theta_{\Z^n}(t)-1,
\end{align}
where the final equality is the sharpness assumption~\eqref{eq:sharp-normalization}.
Equality therefore holds throughout, for every $U\in O(n)$. The first such
equality is a sum of nonnegative terms
\[
  h(z)-e^{-t\norm{z}^2}\ge 0\qquad(z\in U\Z^n\setminus\{0\}),
\]
and the second is a sum of nonpositive terms
\[
  \wh h(z)\le 0\qquad(z\in U\Z^n\setminus\{0\}).
\]
Both sums are absolutely convergent because $h$ and $\wh h$ are Schwartz.
Therefore every term vanishes:
\begin{equation}\label{eq:saturation-on-rotation}
  h(z)=e^{-t\norm{z}^2},\qquad \wh h(z)=0
  \qquad(z\in U\Z^n\setminus\{0\}).
\end{equation}

Now let $x\in\R^n$ be nonzero with $\norm{x}^2\in\Z$. By Lagrange's
four-square theorem, we may choose $z\in\Z^n$ with $\norm{z}=\norm{x}$. There
exists $U\in O(n)$ with $Uz=x$, and so \eqref{eq:saturation-on-rotation} yields
\[
  h(x)=e^{-t\norm{x}^2},\qquad \wh h(x)=0
  \qquad(x\ne 0,\ \norm{x}^2\in\Z),
\]
as claimed in~\eqref{eq:saturation-conclusion}.
\end{proof}

\section{The $E_8$ Obstruction}\label{sec:e8-obstruction}

We now derive Corollary~\ref{cor:strict} from Theorem~\ref{thm:saturation} by
applying the saturation conclusion to the lattice $E_8 \oplus \Z^{n-8}$.

\begin{proof}[Proof of Corollary~\ref{cor:strict}]
Let $h$ satisfy conditions \textup{(i)} and \textup{(ii)} of
Theorem~\ref{thm:saturation}. The inequality
\begin{equation}\label{eq:weak-bound}
  1+\wh h(0)-h(0)\ge \Theta_{\Z^n}(t)
\end{equation}
follows from the upper part of the chain~\eqref{eq:saturation-chain}, applied
with $U=I$. Assume for the sake of seeking a contradiction that equality holds
in~\eqref{eq:weak-bound}. Then $h$ is sharp at $\Z^n$, and by
Theorem~\ref{thm:saturation},
\begin{equation}\label{eq:integer-shell-saturation}
  h(x)=e^{-t\norm{x}^2},\qquad \wh h(x)=0
  \qquad(x\ne 0,\ \norm{x}^2\in\Z).
\end{equation}

We set
\[
  \Lambda=E_8\oplus\Z^{n-8},
\]
with the convention $\Z^0=\{0\}$ when $n=8$. The lattice $E_8$ is even
unimodular, and $\Z^{n-8}$ is unimodular and integral; hence their orthogonal
direct sum is unimodular and integral. Moreover every nonzero $x\in\Lambda$ has
$\norm{x}^2\in\Z_{>0}$. By \eqref{eq:integer-shell-saturation},
\begin{equation}\label{eq:saturation-on-E8sum}
  h(x)=e^{-t\norm{x}^2},\qquad \wh h(x)=0
  \qquad(x\in\Lambda\setminus\{0\}).
\end{equation}
Applying Poisson summation to $\Lambda=\Lambda^*$ and using
\eqref{eq:saturation-on-E8sum}, we find
\begin{align}\label{eq:E8sum-collapse}
  \Theta_\Lambda(t)-1
   &=\sum_{x\in\Lambda\setminus\{0\}} e^{-t\norm{x}^2} \nonumber\\
   &=\sum_{x\in\Lambda\setminus\{0\}}h(x) \nonumber\\
   &=\sum_{x\in\Lambda}h(x)-h(0) \nonumber\\
   &=\sum_{x\in\Lambda}\wh h(x)-h(0) \nonumber\\
   &=\wh h(0)-h(0)
    =\Theta_{\Z^n}(t)-1,
\end{align}
where the last equality is the assumed equality case of \eqref{eq:weak-bound}.
Therefore we have
\begin{equation}\label{eq:theta-E8sum-equals-Zn}
  \Theta_{E_8\oplus\Z^{n-8}}(t)=\Theta_{\Z^n}(t);
\end{equation}
by multiplicativity \eqref{eq:multiplicative-theta}, this says
\[
  \Theta_{E_8}(t)\Theta_{\Z^{n-8}}(t)
  =\Theta_{\Z^8}(t)\Theta_{\Z^{n-8}}(t).
\]
Since $\Theta_{\Z^{n-8}}(t)>0$, we can divide to obtain
$\Theta_{E_8}(t)=\Theta_{\Z^8}(t)$, contradicting Lemma~\ref{lem:e8-gap}.
\end{proof}

Corollary~\ref{cor:nogo} is immediate, as any candidate sharp certificate would
violate~\eqref{eq:strict}.

\section{Extensions}\label{sec:extensions}

The saturation mechanism of Theorem~\ref{thm:saturation} is more flexible than
its single-function statement suggests. In this section we record three
extensions: first from unimodular integral lattices to general stable lattices;
then from a single lattice-independent certificate to any orbit-constant graded
family; and finally from sharp certificates to near-sharp sequences.

\subsection{Stable lattices}\label{sec:stable}

The Regev--Stephens-Davidowitz conjecture~\cite{RegevSDoriginal} was originally
formulated in the broader setting of \emph{stable lattices}, i.e., full-rank
lattices $\Lambda\subset\R^n$ such that $\covol(\Lambda)=1$ and every nonzero
sublattice $\Lambda'\subseteq\Lambda$ has covolume at least $1$ in its linear
span---or equivalently, if $v_1,\ldots,v_r$ is a basis of $\Lambda'$, then
\[
  \sqrt{\det(\ip{v_i}{v_j})_{i,j=1}^r}\ge 1.
\]
Every unimodular integral lattice is stable: the Gram matrix of any nonzero
sublattice of an integral lattice has integer entries and positive determinant,
hence determinant at least $1$, so the covolume of the sublattice in its span
is at least $1$. In particular, $\Z^n$, every rotation $U\Z^n$, and
$E_8\oplus\Z^{n-8}$ are stable, since they are unimodular integral.

For a stable lattice that is not self-dual, the natural scalar Poisson
certificate places the majorization condition on $\Lambda$ and the Fourier
nonpositivity condition on~$\Lambda^*$.

\begin{corollary}[Strictness over stable lattices]\label{cor:stable-strict}
The conclusions of Theorem~\ref{thm:saturation}, Corollary~\ref{cor:strict},
and Corollary~\ref{cor:nogo} continue to hold if conditions \textup{(i)} and
\textup{(ii)} are strengthened to require
\[
  h(x)\ge e^{-t\norm{x}^2}\qquad(x\in\Lambda\setminus\{0\})
\]
and
\[
  \wh h(\xi)\le 0\qquad(\xi\in\Lambda^*\setminus\{0\})
\]
for every stable lattice $\Lambda\subset\R^n$.
\end{corollary}

\begin{proof}
Every rotation $U\Z^n$ is integral and unimodular, since orthogonal maps
preserve inner products and covolume; hence we have $(U\Z^n)^*=U\Z^n$.
Likewise, $E_8\oplus\Z^{n-8}$ is integral and unimodular, and therefore
self-dual. Thus the strengthened stable-lattice hypotheses imply the hypotheses
used in the proofs of Theorem~\ref{thm:saturation} and
Corollary~\ref{cor:strict} at precisely the lattices needed.
\end{proof}

Conceptually, what the saturation argument requires is $O(n)$-invariance of the
constraint class together with the presence of $E_8\oplus\Z^{n-8}$ within it;
integrality is sufficient but not necessary. Thus, the obstruction we have
identified covers the natural scalar Poisson certificate strategy for the
broader stable-lattice formulation of \eqref{eq:RSD} as well.

\subsection{Orbit-constant graded families}\label{sec:graded}

A natural way to attempt to bypass Theorem~\ref{thm:saturation} is to allow the
auxiliary function to depend on the target lattice itself. Concretely, instead
of a single Schwartz function $h$, one specifies a family $\{h_\Lambda\}_\Lambda$
indexed by stable lattices, with $h_\Lambda$ witnessing the Poisson-summation
bound on $\Theta_\Lambda(t)$ specifically. Such a family must depend on
$\Lambda$ through some notion of lattice structure. One natural condition is
\emph{orbit-constancy}:
\[
  h_{U\Lambda} = h_\Lambda\quad\text{as functions on $\R^n$, for every $U\in O(n)$ and every stable $\Lambda$.}
\]
Orbit-constancy is intentionally stronger than the more usual \emph{equivariance}
condition $h_{U\Lambda}(x)=h_\Lambda(U^{-1}x)$; it captures certificate rules
that assign a single ambient auxiliary function to every isometric copy of a
lattice, for example purely radial rules whose parameters are determined by
orientation-independent scalar invariants of $\Lambda$. The saturation
mechanism extends to the orbit-constancy setting.

\begin{theorem}[Strictness for orbit-constant graded certificates]\label{thm:graded-strict}
Let $n\ge 8$ and $t>0$. Suppose that for every stable lattice
$\Lambda\subset\R^n$ we have an even Schwartz function $h_\Lambda:\R^n\to\R$
satisfying:
\begin{enumerate}
\item[\textup{(a)}] $h_{U\Lambda}=h_\Lambda$ for every $U\in O(n)$ and every
stable $\Lambda$;
\item[\textup{(b)}] $h_\Lambda(x)\ge e^{-t\norm{x}^2}$ for every nonzero
$x\in\Lambda$ and every stable $\Lambda$;
\item[\textup{(c)}] $\wh h_\Lambda(\xi)\le 0$ for every nonzero
$\xi\in\Lambda^*$ and every stable $\Lambda$.
\end{enumerate}
If $1+\wh h_{\Z^n}(0)-h_{\Z^n}(0)=\Theta_{\Z^n}(t)$, then, for
$\Lambda_c=E_8\oplus\Z^{n-8}$,
\[
  1+\wh h_{\Lambda_c}(0)-h_{\Lambda_c}(0)>\Theta_{\Z^n}(t).
\]
\end{theorem}

In particular, no orbit-constant graded scalar Poisson family can both be sharp
at $\Z^n$ and establish the sharp bound uniformly over all stable $\Lambda$,
since such a uniform bound would require
\[
  1+\wh h_\Lambda(0)-h_\Lambda(0)\le \Theta_{\Z^n}(t)
\]
for every stable $\Lambda$, including $\Lambda_c=E_8\oplus\Z^{n-8}$.

Orbit-constancy is stronger than asking that $h_\Lambda$ depend only on the
isomorphism class of $\Lambda$---it asks moreover that no information about the
embedding be used in placing the function on $\R^n$. Theorem~\ref{thm:graded-strict}
rules out, in particular, orbit-constant theta-graded families (where the same
ambient function is assigned whenever $\Theta_\Lambda$, equivalently the
shell-count sequence, is the same); orbit-constant modular-form-graded families
(where $\Lambda$ is even unimodular and $h_\Lambda$ depends only on the
coordinates of $\Theta_\Lambda$ in a fixed basis of
$M_{n/2}(\mathrm{SL}_2(\Z))$); and orbit-constant finite-parameter intrinsic
families (where $h_\Lambda$ depends on finitely many $O(n)$-invariants of
$\Lambda$). Equivariant families satisfying only
$h_{U\Lambda}(x)=h_\Lambda(U^{-1}x)$ are \emph{not} ruled out by
Theorem~\ref{thm:graded-strict}; we return to this point in Section~\ref{sec:scope}.

\begin{proof}[Proof of Theorem~\ref{thm:graded-strict}]
Suppose for the sake of seeking a contradiction that
\[
  1+\wh h_{\Lambda_c}(0)-h_{\Lambda_c}(0)\le \Theta_{\Z^n}(t).
\]

\smallskip
\emph{Step 1: $h_{\Z^n}$ equals the Gaussian on every integer-squared-radius
shell.} Fix $U\in O(n)$. The lattice $U\Z^n$ is stable and self-dual, satisfies
$\Theta_{U\Z^n}(t)=\Theta_{\Z^n}(t)$, and by~\textup{(a)} has
$h_{U\Z^n}=h_{\Z^n}$. Since $U\Z^n = (U\Z^n)^*$, condition~\textup{(c)}
applies to the same nonzero points $z \in U\Z^n$ as condition~\textup{(b)}.
Applying the saturation chain~\eqref{eq:saturation-chain} with $h$ replaced by
$h_{\Z^n}$ on $U\Z^n$, using~\textup{(b)} and~\textup{(c)} on $U\Z^n$, and
the sharpness assumption at $\Z^n$, every inequality in the chain is an
equality. By the absolute-convergence argument used in the proof of
Theorem~\ref{thm:saturation},
\[
  h_{\Z^n}(z)=e^{-t\norm{z}^2},\qquad \wh h_{\Z^n}(z)=0
  \qquad(z\in U\Z^n\setminus\{0\}).
\]
By Lagrange's four-square theorem, every nonzero $x\in\R^n$ with
$\norm{x}^2\in\Z$ has the form $Uz$ for some $z\in\Z^n$ and $U\in O(n)$, so
\begin{equation}\label{eq:graded-shell-saturation}
  h_{\Z^n}(x)=e^{-t\norm{x}^2},\qquad \wh h_{\Z^n}(x)=0
  \qquad(x\ne 0,\ \norm{x}^2\in\Z).
\end{equation}

\smallskip
\emph{Step 2: a Poisson inequality between $h_{\Z^n}$ and a competitor.}
Set $\Lambda_c=E_8\oplus\Z^{n-8}$---this lattice is stable, integral, and
unimodular; hence $\Lambda_c=\Lambda_c^*$, and every nonzero $x\in\Lambda_c$
has $\norm{x}^2\in\Z_{>0}$. Define $F=h_{\Lambda_c}-h_{\Z^n}$. Then $F$ is
even Schwartz, and by~\textup{(b)},~\textup{(c)}, the self-duality
$\Lambda_c=\Lambda_c^*$, and~\eqref{eq:graded-shell-saturation},
\begin{equation}\label{eq:F-signs}
  F(x)\ge 0,\qquad \wh F(x)\le 0
  \qquad(x\in\Lambda_c\setminus\{0\}).
\end{equation}
Poisson summation applied to $F$ on $\Lambda_c=\Lambda_c^*$ gives
\[
  F(0)+\sum_{x\in\Lambda_c\setminus\{0\}}F(x)
  =\wh F(0)+\sum_{x\in\Lambda_c\setminus\{0\}}\wh F(x),
\]
which rearranges to
\begin{equation}\label{eq:F-comparison}
  \wh F(0)-F(0)
  =\sum_{x\in\Lambda_c\setminus\{0\}}F(x)
   -\sum_{x\in\Lambda_c\setminus\{0\}}\wh F(x).
\end{equation}
By~\eqref{eq:F-signs}, every term on the right-hand side of
\eqref{eq:F-comparison} is nonnegative, so $\wh F(0)-F(0)\ge 0$. Unpacking, we
obtain
\[
  1+\wh h_{\Lambda_c}(0)-h_{\Lambda_c}(0)
  \ge 1+\wh h_{\Z^n}(0)-h_{\Z^n}(0)
  =\Theta_{\Z^n}(t).
\]

\smallskip
\emph{Step 3: equality forces a theta contradiction.} By assumption
$1+\wh h_{\Lambda_c}(0)-h_{\Lambda_c}(0)\le\Theta_{\Z^n}(t)$, while Step~2
gives the reverse inequality; hence we must have $\wh F(0)-F(0)=0$.
Thus~\eqref{eq:F-comparison} gives
\[
  \sum_{x\in\Lambda_c\setminus\{0\}}F(x)
  +\sum_{x\in\Lambda_c\setminus\{0\}}\bigl(-\wh F(x)\bigr)
  =0.
\]
Both sums are absolutely convergent sums of nonnegative terms, so it must be
that every term vanishes:
\[
  h_{\Lambda_c}(x)=e^{-t\norm{x}^2},\qquad \wh h_{\Lambda_c}(x)=0
  \qquad(x\in\Lambda_c\setminus\{0\}).
\]
Poisson summation applied to $h_{\Lambda_c}$ on $\Lambda_c=\Lambda_c^*$ then
gives
\begin{align*}
  \Theta_{\Lambda_c}(t)-1
   &=\sum_{x\in\Lambda_c\setminus\{0\}}e^{-t\norm{x}^2} \\
   &=\sum_{x\in\Lambda_c\setminus\{0\}}h_{\Lambda_c}(x) \\
   &=\wh h_{\Lambda_c}(0)-h_{\Lambda_c}(0)
    +\sum_{x\in\Lambda_c\setminus\{0\}}\wh h_{\Lambda_c}(x)\\
   &=\Theta_{\Z^n}(t)-1.
\end{align*}
Hence $\Theta_{\Lambda_c}(t)=\Theta_{\Z^n}(t)$, and by
\eqref{eq:multiplicative-theta}, $\Theta_{E_8}(t)=\Theta_{\Z^8}(t)$,
contradicting Lemma~\ref{lem:e8-gap}.
\end{proof}

\subsection{Near-sharp sequences}\label{sec:near-sharp}

Our saturation argument is also stable under approximation: any near-sharp
certificate must asymptotically interpolate the Gaussian and annihilate its
Fourier transform on every positive integral shell.

\begin{proposition}[Approximate saturation]\label{prop:approx-saturation}
Let $n\ge 4$ and $t>0$. Let $\{h_j\}\subset\calS(\R^n)$ be a sequence of even
Schwartz functions, each satisfying conditions \textup{(i)} and \textup{(ii)}
of Theorem~\ref{thm:saturation}. Set
\[
  \varepsilon_j=\varepsilon_j(h_j):=1+\wh h_j(0)-h_j(0)-\Theta_{\Z^n}(t)\ge 0
\]
and suppose that $\varepsilon_j\to 0$. Then for every nonzero $x\in\R^n$ with
$\norm{x}^2\in\Z$,
\[
  0\le h_j(x)-e^{-t\norm{x}^2}\le \varepsilon_j
  \quad\text{and}\quad
  0\le -\wh h_j(x)\le \varepsilon_j.
\]
In particular,
\[
  h_j(x)\to e^{-t\norm{x}^2}
  \quad\text{and}\quad
  \wh h_j(x)\to 0
\]
on every nonzero positive-integer-norm shell.
\end{proposition}

\begin{proof}
Fix $U\in O(n)$ and apply the chain~\eqref{eq:saturation-chain} to $h_j$ on
$U\Z^n$. Writing the slack in the two inequalities of
\eqref{eq:saturation-chain} as
\[
  A_j^U:=\sum_{z\in U\Z^n\setminus\{0\}}\bigl(h_j(z)-e^{-t\norm{z}^2}\bigr)\ge 0,
  \qquad
  B_j^U:=-\sum_{z\in U\Z^n\setminus\{0\}}\wh h_j(z)\ge 0,
\]
the chain becomes
\[
  \Theta_{\Z^n}(t)-1+A_j^U+B_j^U
  =\wh h_j(0)-h_j(0)
  =\Theta_{\Z^n}(t)-1+\varepsilon_j,
\]
so $A_j^U+B_j^U=\varepsilon_j$. Each summand of $A_j^U$ is nonnegative, so
each is at most $\varepsilon_j$; similarly for $B_j^U$. Choosing $U$ to carry a
vector $z\in\Z^n$ of squared norm $\norm{x}^2$ to $x$ yields the stated bounds
at $x$.
\end{proof}

\begin{corollary}[No compact limiting certificates]\label{cor:no-compact-limits}
Let $n\ge 8$, $t>0$, and $\Lambda=E_8\oplus\Z^{n-8}$. Then there is no sequence
$\{h_j\}\subset\calS(\R^n)$ of even Schwartz functions satisfying conditions
\textup{(i)} and \textup{(ii)} of Theorem~\ref{thm:saturation} with
$\varepsilon_j(h_j)\to 0$ such that the restrictions of $h_j$ and $\wh h_j$ to
$\Lambda$ are uniformly absolutely summable, i.e., such that there exist
summable nonnegative families $(A_x)_{x\in\Lambda}$ and $(B_x)_{x\in\Lambda}$
for which
\[
  \abs{h_j(x)}\le A_x,
  \qquad
  \abs{\wh h_j(x)}\le B_x
\]
for all $x\in\Lambda$ and all $j$.
\end{corollary}

\begin{proof}
By Proposition~\ref{prop:approx-saturation}, for each nonzero $x\in\Lambda$,
\[
  h_j(x)\to e^{-t\norm{x}^2}
  \quad\text{and}\quad
  \wh h_j(x)\to 0.
\]
By dominated convergence under the uniform summability hypothesis,
\[
  \sum_{x\in\Lambda\setminus\{0\}}h_j(x)\to\Theta_\Lambda(t)-1
  \quad\text{and}\quad
  \sum_{x\in\Lambda\setminus\{0\}}\wh h_j(x)\to 0.
\]
Poisson summation on $\Lambda=\Lambda^*$ gives
\[
  \sum_{x\in\Lambda\setminus\{0\}}h_j(x)
  =\wh h_j(0)-h_j(0)
  +\sum_{x\in\Lambda\setminus\{0\}}\wh h_j(x).
\]
Letting $j\to\infty$ and using
$\wh h_j(0)-h_j(0)\to\Theta_{\Z^n}(t)-1$, we obtain
$\Theta_\Lambda(t)-1=\Theta_{\Z^n}(t)-1$, contradicting Lemma~\ref{lem:e8-gap}.
\end{proof}

Corollary~\ref{cor:no-compact-limits} should be interpreted in part as a
limitation of the present obstruction. Theorem~\ref{thm:saturation} rules out
attained certificates, and Corollary~\ref{cor:no-compact-limits} rules out
limiting schemes with enough compactness or uniform summability to pass Poisson
summation to the limit on $E_8\oplus\Z^{n-8}$. It does not by itself exclude a
genuinely noncompact limiting procedure or a distributional certificate class in
which pointwise lattice values are no longer available without additional
regularity assumptions.

\section{Remarks on Scope}\label{sec:scope}

\subsection*{Scope of the obstruction}
The obstruction begins at $n=8$, the first dimension in which the unimodular
integral comparison is nontrivial. Indeed, the classical classification of
positive definite unimodular integral lattices shows that $\Z^n$ is the unique
such lattice up to isometry for $n\le 7$, while in dimension $8$ one has both
the odd lattice $\Z^8$ and the even lattice $E_8$ \cite[Chs.~4,~16]{ConwaySloane}.

The obstruction is also sharp in its target. Corollaries~\ref{cor:strict} and
\ref{cor:nogo} rule out certificates that prove the exact $\Z^n$ constant by
the scalar--Poisson-summation argument; Theorem~\ref{thm:graded-strict} rules
out the same for orbit-constant graded families; and
Corollary~\ref{cor:no-compact-limits} rules out compact or dominated limiting
schemes. We do not rule out weaker numerical upper bounds, methods using
additional structure of a given lattice, or genuinely noncompact limiting
procedures whose limit object is not a Schwartz certificate of the required
form.

\subsection*{Why the Gaussian itself is not a certificate}
The function being bounded is the Gaussian, but the Gaussian is not a
Cohn--Elkies-style certificate. If $g_t(x)=e^{-t\norm{x}^2}$, then $g_t$
satisfies the majorization condition \textup{(i)} with equality. However,
\[
  \wh g_t(\xi)=\left(\frac{\pi}{t}\right)^{n/2}
  e^{-\pi^2\norm{\xi}^2/t}>0
  \qquad(\xi\in\R^n),
\]
so $g_t$ fails the required nonpositivity condition \textup{(ii)} at every
nonzero frequency. Our results rule out a sharp Poisson certificate for the
Gaussian mass; they do not suggest that the Gaussian is the wrong test function
in the Regev--Stephens-Davidowitz conjecture.

\subsection*{Strategies beyond the obstruction}
Our obstruction concerns even Schwartz certificates that are either
lattice-independent or, in the graded version, assigned orbit-constantly to the
target lattice. Several natural approaches remain outside that method locus.

First, theta series of unimodular lattices are modular forms, and a sizable
literature exploits the structure of modular form spaces. The works of
Belfiore--Oggier, Belfiore--Sol\'e, Oggier--Sol\'e--Belfiore,
Ernvall-Hyt\"onen, Lin--Oggier, and Bollauf--Lin study secrecy functions,
theta-series ratios, flatness factors, and related modular-form
parametrizations at fixed argument, rather than a single Poisson majorant
\cite{BelfioreOggier,BelfioreSole,OggierSoleBelfiore,ErnvallHytonen,LinOggier,BollaufLin,BollaufLinYtrehus}.
When such parametrizations are orbit-constant (depending only on
$\Theta_\Lambda$ and placed canonically), they are covered by
Theorem~\ref{thm:graded-strict}; the methods themselves, applied at fixed
argument rather than to derive a global pointwise bound, are not.

Second, a natural successor framework is to seek \emph{higher-order or
non-scalar} relaxations. Our results rule out a single scalar auxiliary
function per $O(n)$-orbit. They do not address methods depending on
configurations of several lattice vectors, such as their Gram matrices.
Integral lattices satisfy $\ip{x}{y}\in\Z$, a constraint on pairs of vectors
that is invisible to a single-point scalar certificate. The Bachoc--Vallentin
semidefinite programming framework \cite{BachocVallentin} and its
harmonic-analytic refinements retain such angular information: variables are
indexed by triples $(\norm{y_1}^2,\norm{y_2}^2,\ip{y_1}{y_2})$ of squared
lengths and inner products, and positive semidefiniteness constraints arise
from Gegenbauer/Schoenberg expansions of spherical-harmonic kernels. In the
integral-lattice setting, the inner product is constrained to $\Z$, restricting
the SDP support to a sparse discrete set of cosines that scalar certificates
structurally cannot access.

Third, as noted in Section~\ref{sec:graded}, \emph{equivariant non-constant graded
families} satisfying $h_{U\Lambda}(x)=h_\Lambda(U^{-1}x)$, rather than the stronger
orbit-constancy hypothesis $h_{U\Lambda}=h_\Lambda$, are outside the scope of
Theorem~\ref{thm:graded-strict}. In such a family $h_\Lambda$ genuinely depends
on the embedding of $\Lambda$, and Step~1 of the proof of
Theorem~\ref{thm:graded-strict} propagates saturation of $h_{\Z^n}$ only to
$\Z^n\setminus\{0\}$, rather than to the full integer-norm locus. Such families
amount to solving a separate scalar Poisson linear program for each embedded
lattice, so the present saturation argument does not distinguish them from a
direct attack on the conjectural comparison itself.

Fourth, low-dimensional or otherwise restricted cases may be accessible by
classification. In dimensions where unimodular lattices are classified, the
conjectural inequality can be reduced to finitely many theta-series
comparisons. This route is orthogonal to the universal auxiliary-function
strategy ruled out here.

\subsection*{Relation to the Belfiore--Sol\'e conjecture}
In the secrecy-gain literature, with the convention
$\Theta_\Lambda(z)=\sum_{x\in\Lambda}e^{\pi i z\norm{x}^2}$ for
$\operatorname{Im}z>0$, the secrecy function of a unimodular lattice is
$\Xi_\Lambda(y)=\Theta_{\Z^n}(iy)/\Theta_\Lambda(iy)$ for $y>0$. Under the
substitution $t=\pi y$, this connects to the present Gaussian-mass convention as
$\Theta_\Lambda(iy)=\Theta_\Lambda(t)$ with $t=\pi y$. The Belfiore--Sol\'e
conjecture asserts that $\Xi_\Lambda$ attains its global maximum at the symmetry
point $y=1$ for unimodular lattices \cite{BelfioreSole}. This is a single-point
assertion about the location of an extremum; the Regev--Stephens-Davidowitz
conjecture \eqref{eq:RSD} is instead the global pointwise inequality
$\Theta_\Lambda(t)\le\Theta_{\Z^n}(t)$ for all $t>0$. Our results apply to
attempts to prove the global pointwise inequality through a scalar Poisson
certificate; they do not obstruct the single-point or symmetry-point methods
used in the secrecy-gain setting.

\subsection*{Nonunimodular lattices}
Our results do not claim to cover the full nonunimodular integral case. If
$\Lambda$ is integral but not unimodular, then $\Lambda\subsetneq\Lambda^*$ and
$\covol(\Lambda)>1$, and Poisson summation reads as
\[
  \sum_{x\in\Lambda}h(x)=\frac{1}{\covol(\Lambda)}
  \sum_{\xi\in\Lambda^*}\wh h(\xi);
\]
thus the primal and dual point sets no longer coincide, and the factor
$1/\covol(\Lambda)$ changes the equality bookkeeping. Our saturation argument
depends essentially on the self-dual normalization $\Lambda=\Lambda^*$, in
which the same nonzero lattice points carry the majorization and Fourier-sign
conditions. Whether a comparable obstruction can be formulated for a natural
nonunimodular certificate class is left open.

\section{Conclusion}
The obstruction proven here is deliberately narrow but rigid. It does not
settle the underlying Gaussian mass comparison, nor does it rule out all
linear-programming or Cohn--Elkies-type methods. Rather, it identifies a
specific failure mode for the most direct scalar Poisson-summation certificate,
in the orbit-constant formulations considered: a single universal function,
the corresponding stable-lattice certificate, orbit-constant graded families,
and compact near-sharp limits. In dimensions $n\ge 8$, the saturation forced by
sharpness at $\Z^n$ is incompatible with the strict structural theta-series gap
between $\Z^8$ and $E_8$. If a Poisson-summation proof of the sharp Gaussian
mass bound exists, it must use either a genuinely equivariant non-constant
construction, a higher-order or multi-point framework, or a formulation beyond
the single-Schwartz-function scalar framework considered here.

\newcommand{\etalchar}[1]{$^{#1}$}
\providecommand{\bysame}{\leavevmode\hbox to3em{\hrulefill}\thinspace}
\providecommand{\MR}{\relax\ifhmode\unskip\space\fi MR }
\providecommand{\MRhref}[2]{%
  \href{http://www.ams.org/mathscinet-getitem?mr=#1}{#2}
}

\end{document}